\documentclass[11pt,authoryear,letterpaper,reqno,oneside]{amsart}
\usepackage[left=1.2in,right=1.2in,bottom=1.2in,top=1.2in]{geometry}
\usepackage[bookmarks, colorlinks = true, citecolor = blue, linkcolor = blue, urlcolor = blue]{hyperref}
\usepackage{
	amsfonts,
	amsthm,
	amsmath,
	amssymb,
	dcolumn,
	epsfig,
	epstopdf,
	graphicx,
	latexsym,
	mathrsfs,
	multicol,
	textcomp,
	url,
	xspace
}
\usepackage[all]{xy}
\newtheorem{theorem}{Theorem}[section]

\newtheorem{fact}[theorem]{Fact}

\newtheorem{prop}[theorem]{Proposition}
\newtheorem{claim}[theorem]{Claim}

\newtheorem{definition}[theorem]{Definition}

\newtheorem{lemma}[theorem]{Lemma}

\newtheorem{cor}[theorem]{Corollary}

\newcommand{\vectornorm}[1]{\left| \! \left|#1\right| \! \right|}
\newcommand{\rapf}{$\RA$:\ }
\newcommand{\lapf}{$\LA$:\ }
\newcommand{\lland}{\ \land \ }
\newcommand{\eps}{\epsilon}

\newcommand{\NN}{{\mathbb{N}}}
\newcommand{\RR}{{\mathbb{R}}}

\newcommand{\sub}{\subseteq}
\newcommand{\sN}[1]{_{#1\in \NN}}
\newcommand{\uhr}[1]{\! \upharpoonright_{#1}}
\newcommand{\ML}{Martin-L{\"o}f}
\newcommand{\SI}[1]{\Sigma^0_{#1}}

\newcommand{\bi}{\begin{itemize}}
\newcommand{\ei}{\end{itemize}}
\newcommand{\bc}{\begin{center}}
\newcommand{\ec}{\end{center}}

\newcommand{\ria}{\rightarrow}
\newcommand{\tp}[1]{2^{#1}}
\newcommand{\ex}{\exists}

\newcommand{\seqcantor}{2^{ \omega}}
\newcommand{\strcantor}{2^{< \omega}}
\newcommand{\cantor}{\seqcantor}

\newcommand{\Opcl}[1]{[#1]^\prec}

\newcommand{\Om}{\Omega}
\newcommand{\n}{\noindent}

\newcommand{\vsps}{\vspace{3pt}}

\newcommand{\vsp}{\vspace{6pt}}
\newcommand{\leb}{\mathbf{\lambda}}

\newcommand{\sss}{\sigma}
\newcommand{\aaa}{\alpha}

\newcommand\+[1]{\mathcal{#1}}

\newcommand{\R}{{\mathbb R}}
\newcommand{\LR}{\Leftrightarrow}

\newcommand{\RA}{\Rightarrow}

\newcommand{\LA}{\Leftarrow}
\newcommand{\ol}{\overline}
\newcommand{\ul}{\underline}

\newcommand{\Fcn}{\mbox{\rm \textsf{cdf}}}
\newcommand{\DA}{\downarrow}

\newcommand{\UM}{\mathbb{U}}
\newcommand{\sssl}{\ensuremath{|\sigma|}}

\def\Hawaii{Hawai\kern.05em`\kern.05em\relax i}

\begin{document}

\title{Algorithmic  aspects of Lipschitz functions}
\author[
Freer, Kjos-Hanssen, Nies and Stephan]
{Cameron Freer, Bj\o rn Kjos-Hanssen, Andr\'e Nies and Frank Stephan}
\address{C.~Freer, Computer Science and Artificial Intelligence Laboratory,
Massachusetts Institute of Technology,
32 Vassar Street,
Cambridge, MA 02139 }
\email{freer@math.mit.edu}
\address{B.~Kjos-Hanssen, Department of Mathematics, University of \Hawaii\ at M\=anoa, 2565 McCarthy Mall, Honolulu, HI 96822 }
\email{bjoern@math.hawaii.edu}
\address{A.~Nies, Department of Computer Science, University of Auckland,
 Private Bag 92019, Auckland, New Zealand}
\email{andre@cs.auckland.ac.nz}
\address{F.~Stephan, Department of Mathematics, National University of Singapore, Singapore 119076, Republic of Singapore}
\email{fstephan@comp.nus.edu.sg}

	\begin{abstract}
		We characterize the variation functions of com\-pu\-ta\-ble Lipschitz functions. We show that a real $z$ is computably random if and only if every
		computable Lipschitz function is differentiable at $z$. Beyond these principal results, we show that a real $z$ is Schnorr random if and only if every
		Lipschitz function with $L_1$-computable derivative is differentiable at $z$.
	\end{abstract}
	
\maketitle
	
	\section{Introduction}
		Lipschitz functions are of fundamental importance in analysis. They appear naturally in various contexts, such as the solvability of differential equations.
		For a set $A \sub \RR^n$, recall that a function $f\colon A \to \RR^m$ is \emph{Lipschitz}
		if there is a constant~$c$, called a Lipschitz bound, such that for all $x$ and $y$ we have
		$\vectornorm {f(x)-f(y)} \le c \vectornorm{x-y}$ (say, $\vectornorm {\cdot} $ denotes the Euclidean norm).
		There are many theorems stating that Lipschitz functions are in one sense or another well-behaved.
		For instance, the McShane--Whitney extension theorem says that
		a Lipschitz function $f\colon A \to \RR^m$ can be extended to a Lipschitz function, with the same least Lipschitz bound, that is defined on all of $\RR^n$.
		Rademacher's theorem states that $f$ is differentiable at almost every point in $A$.
		In dimensions $n=m=1$ this is immediate from the well-known theorem of Lebesgue that every real function of bounded variation is differentiable almost everywhere.
		In higher dimensions one uses arguments particular to Lipschitz functions.
		See \cite{Heinonen:04} and the references given there for more background on Lipschitz functions.

		Computable analysis seeks algorithmic analogues of theorems from analysis when effectiveness conditions are imposed on the functions.
		Several important theorems from analysis state almost everywhere well-behavior of functions in certain classes.
		In recent years, such theorems have been studied intensely from the point of view of algorithmic randomness~\cite{Pathak:09,Brattka.Miller.ea:nd,Pathak.Rojas.ea:12,Rute:12};
		this provides a way of understanding the complexity of exception null sets, and to characterize algorithmic randomness notions via computable analysis.
		Much earlier, the constructivist Demuth already observed this connection in papers such as \cite{Demuth:75}.
		Our purpose is to carry out some of this program in the setting of Lipschitz analysis.

		We briefly discuss some basic concepts in computable analysis.
		A sequence $(q_n)\sN n$ of rationals is called a \emph{Cauchy name} if $|q_{n} - q_{k} | \le \tp{-n}$ for each $k \ge n$.
		If $\lim_n q_n =x$ we say that $(q_n)\sN n$ is a \emph{Cauchy name for} $x$. Thus, $q_n$ approximates~$x$ up to an error $|x - q_n | $ of at most $ \tp{-n}$.
		A real $x$ is called \emph{computable} if it has a computable Cauchy name.

		In Subsection~\ref{ss:compfunctions} we will in detail discuss computability of functions defined on the unit interval.
		For now, it suffices to know that a Lipschitz function~$f$ is computable if and only if
		$f(q)$ is a computable real uniformly in a rational $q \in [0,1]$.
		This is not the original definition of computability for functions defined on $[0,1]$; see the discussion after Definition \ref{CompDef}.
		The same condition defines computability of a continuous monotonic function.

		A real $x$ is called \emph{left-r.e.} (short for left-recursively enumerable) if the set of rationals less than $x$ is recursively enumerable.
		Equivalently, there is an increasing computable sequence of rationals $(q_n)\sN n$ such that $x = \sup_n q_n$.

		In the following subsections we will give an overview of our main results.
		Possible extensions and open questions will be discussed in the concluding section.

		\subsection{The variation of a computable Lipschitz function}
			Let $g\colon [0,1] \ria \R$. For $0 \le x< y \le 1$ recall the \emph{variation} of $g$ in $[x,y]$:
			\[
				V(g,[x,y]) = \sup \left\{\sum_{i=1}^{n-1} \bigl| g(t_{i+1}) - g(t_i)\bigr| : x \le t_1 \le t_2 \le \ldots \le t_n \le y\right\}.
			\]
			We have $V(g,[x,y]) + V(g, [y,z]) = V (g, [x,z])$ for $x< y< z$ (see \cite[Prop.\ 5.2.2]{Bogachev.vol1:07}).
			Note that if $g$ is (uniformly) continuous,
			then we may restrict the sequences $t_1 \le t_2 \le \ldots \le t_n$ above to a dense subset of $[0,1]$, such as the dyadic rationals.
			We write $V_g$ for the function $x \mapsto V(g, [0,x])$. Each Lipschitz function is of bounded variation.
			In fact, it is easy to see that $g$ is Lipschitz iff $V_g$ is, and they have the same least Lipschitz constant.

			We will provide a characterization of the class of
			variation functions~$V_g$ for computable Lipschitz functions~$g$.
			If $v$ is any non-decreasing function with $v(0)=0$ then $v=V_v$, so classically,
			the functions~$V_g$ for Lipschitz functions~$g$ are simply the nondecreasing Lipschitz functions starting at the origin.
			In an effective setting, this simple correlation breaks down, because the variation function of a computable Lipschitz function is not necessarily computable.
			In fact, even the total variation $V_g(1)$ of a computable Lipschitz functions $g$ defined on $[0,1]$ need not be
			computable as a real. Note that $V_g(1)$ is always left-r.e. Conversely, we will show in Fact~\ref{fact:easiest} that
			every left-r.e.\ real in $[0,1]$ is of the form $V_g(1)$ for some computable function $g$ with Lipschitz constant~$1$.

			If~$g$ is a computable Lipschitz function and $f=V_g$, then
			$f$ is non-decreasing Lipschitz, we have $f(0)= 0$,
			and $f(y)- f(x)$ is left-r.e.\ uniformly in rationals $x< y$. We call nondecreasing functions~$f$
			satisfying the last two conditions \emph{interval-r.e.} Our first main result, Theorem~\ref{thm:left-r.e._Lipschitz_variation}, shows that this weak
			effectiveness condition on $f$ is sufficient:
				\emph{every Lipschitz interval-r.e.\ function~$f$ is of the form $V_g$ for some computable Lipschitz function $g$}.
		 
			Our proof relies on the following notion. A \emph{signed martingale} is a function $2^{< \omega} \ria \R$ such that the fairness condition
			$ M(\sss 0) + M(\sss 1) = 2 M(\sss)$
			holds for each string $\sss$.
			We proceed via the fact, of interest in itself, that every left-r.e.\ positive martingale with a non-atomic associated measure on Cantor space
			(see Subsection~\ref{ss:martingales}) is the variation martingale of a signed computable martingale.
			The definition of the variation martingale corresponds to the variation measure $|\mu|$ of a signed measure $\mu$.
			Recall that $|\mu | (E)$ is the supremum over all $\sum_i |\mu(E_i)|$ where $(E_i) \sN i$ ranges over partitions of $E$ into measurable sets.
			See, for instance, \cite[Section 6.1]{Rudin:87}.

			After seeing our Theorem~\ref{thm:left-r.e._Lipschitz_variation},
			Jason Rute has provided an extension of our construction to all continuous interval-r.e.\ functions $f$, by showing that $f=V_g$ for some computable~$g$.
			We include this as a theorem joint with Rute at the end of Section~\ref{s:LipVar}.

			One can also ask whether a similar result holds in the context of effective measure theory developed by G\'acs, Hoyrup, Rojas and others (see~\cite{Hoyrup.Rojas:09}).
			For instance, is every lower semi-computable measure on $[0,1]^n$ without point masses the variation,
			in the sense of \cite[Section 6.1]{Rudin:87}, of a computable signed measure?
			Rute has pointed out that our Theorem~\ref{thm:left-r.e._Lipschitz_variation} implies an affirmative answer in the case $n=1$,
			and that the hypothesis to have no point masses is necessary.

		\subsection{Lipschitz functions and computable randomness}
			The following is due to Brattka, Miller, and Nies \cite[Thm.\ 4.1]{Brattka.Miller.ea:nd}:

			\begin{theorem}\label{thm:Brattka}
				Let $z\in [0,1]$. Then
				$z$ is computably random $\LR$ 
				
				\hfill
				each computable nondecreasing function
				$g \colon \, [0,1] \ria \mathbb R$ is differentiable at~$z$.
			\end{theorem}

			\vsps

			\n
			Our second main result, Theorem~\ref{thm:comprd_Lipschitz}, provides  an analogous fact for computable Lipschitz functions.

		\subsection{Lebesgue points and Schnorr randomness}
			Let $L_1([0,1]^n)$ denote the set of integrable functions $g\colon [0,1]^n\ria \R$.
			Recall that a vector $z \in [0,1]^n$ is called a \emph{Lebesgue point} of such a function $g$ if
			\[
				\lim_{z \in Q \lland \leb Q \to 0} (\leb Q)^{-1} \int_Q |g - g(z)| = 0,
			\]
			where $Q$ ranges over $n$-cubes. We say that $z$ is a \emph{weak Lebesgue point} of $g$ if the limit
			\[
				\lim_{z \in Q \lland \leb Q \to 0} (\leb Q)^{-1} \int_Q g
			\]
			exists. The Lebesgue differentiation theorem states the following.
			\begin{theorem} \label{Thm:LebesgueDiff}
				Let $g \in L_1([0,1]^n)$. Then almost every point in $[0,1]^n$ is a Lebesgue point of~$g$.
			\end{theorem}
			For a proof see for instance Rudin \cite[Thm.\ 7.7]{Rudin:87}.
			One can replace the cubes in the definition of Lebesgue points by other geometric objects, such as balls centered at $z$.
			Rudin \cite[Thm.\ 7.10]{Rudin:87} has given a general definition of
			a sequence of Borel sets $(E_k)\sN k$ ``shrinking nicely'' to $z$ that makes the theorem hold; this encompasses both cubes and balls.

			We will formulate the results related to Schnorr randomness (see Subsection~\ref{ss:Schnorr} for a definition) first in terms of weak Lebesgue points.
			In dimension~$1$ they will later on be translated to results on differentiability of Lipschitz functions that are effective in a strong sense.
			\begin{theorem}[Pathak, Rojas, and Simpson \cite{Pathak.Rojas.ea:12} and Rute \cite{Rute:12}]\label{Pathak.Rojas.eaR}
				Let $z\in [0,1]^n$. Then
				$z$ is Schnorr random $\LR$ 
				$z$ is a weak Lebesgue point of every $L_1$-computable function.
			\end{theorem}
			In Theorem~\ref{thm:Lp_Schnorr} we   give a proof  of the implication ``$\LA$" in Theorem \ref{Pathak.Rojas.eaR},
			which we obtained   independently from~\cite{Pathak.Rojas.ea:12}\footnote{
				This and most other results of this paper were presented by Cameron Freer as a contributed talk on March 24, 2011 at the
				2011 North American Annual Meeting of the ASL (held in Berkeley, CA, from March 24 to March 27, almost a year before~\cite{Pathak.Rojas.ea:12} was available).\
				The abstract, entitled ``Effective Aspects of Lipschitz
Functions'', was published in the June 2012 Bull.\ Symb.\ Logic.
			}.
			We show that \emph{
					if $z$ is not Schnorr random, then there is a bounded $L_1$-computable function $g\colon [0,1]^n\ria \R$  
					such  that $z$ is not a weak Lebesgue point of $g$. } We note that the function obtained in~\cite{Pathak.Rojas.ea:12} is also bounded. In comparison,   the proof of Theorem~\ref{thm:Lp_Schnorr} shows that our function does not have $z$ as a weak Lebesgue point because of  a technical property (\ref{eqn:ours_is_better}), which was not achieved in~\cite{Pathak.Rojas.ea:12}. 
			
			As a corollary, we characterise Schnorr randomness in terms of differentiability of Lipschitz functions that are computable in the variation norm.

			We note that in dimension $1$,
			the implication ``$\LA$'' in Theorem \ref{Pathak.Rojas.eaR} can also be derived from the proof of Brattka, Miller, and Nies \cite[Theorem
			6.7]{Brattka.Miller.ea:nd}. In their implication (iii)$\to$(i), given a real $z$ that fails a \ML\ test $(G_m)\sN m$,
			they build a computable function $f$ of bounded variation that is not differentiable at $z$.
			(This result was already announced by
			Demuth \cite{Demuth:75}, albeit in constructive language.) It suffices to observe that, if the given
			test $(G_m)\sN m$ is a Schnorr test, then the function~$g$ with $\int_0^x g = f(x)$
			constructed in \cite[Claim 6.6]{Brattka.Miller.ea:nd} is $L_1$-computable.

		\subsection{Lipschitz functions and Schnorr randomness}
			Since the function $g\colon [0,1]\ria \R$ obtained in Theorem~\ref{thm:Lp_Schnorr} is bounded, the function $f$ given by $f(x) =\int_0^x g$ is Lipschitz,
			and $f'(x)= g(x)$ for almost every $x$.
			Thus, if $z$ is not Schnorr random, we can build a computable Lipschitz function $f$ with $f'$ $L_p$--computable
			for each computable real $p \ge 1$ that is not differentiable at $z$.

			To formulate an appropriate effectiveness condition for $f$ itself rather than $f'$, recall that any Lipschitz function is absolutely continuous.
			For any absolutely continuous function $f$, we have $V(f, [0,x])= \int_0^x |f'|$.
			Then $f'$ is $L_1$-computable iff $f$ is computable in the variation norm, defined by $\vectornorm f _{V} = |f(0)| +V(f, [0,1])$
			This means that there is an effective sequence of rational polynomials $(q_n)\sN n$ such that $\vectornorm {f - q_n}_{V} \le \tp{-n}$.
			The latter condition is stronger than the mere computability of $f$.
			At the end of Subsection~\ref{ss:Lp_comp} we will give somewhat more technical detail.
			Also see \cite[p.\ 376]{Carothers:00} for detail on the variation norm.

			Let us  compare our results on computable randomness and on the weaker notion of Schnorr randomness.
			Note that in each case we pass from a randomness test failed by a real $z$ to a computable Lipschitz function not differentiable at $z$.
			If the real $z$ is
			not computably random, as shown by a computable martingale $M$ that succeeds on $z$, we obtain a computable Lipschitz function $f$ that is not differentiable at $z$.
			If $z$ is not even Schnorr random, as shown by a Schnorr test,
			we obtain a computable Lipschitz function $f$ such that $f'$ is $L_p$-computable for each $p$; in particular, $f$ is computable in the variation norm.

	\section{Preliminaries}\label{s:prelims}
	We collect some background and definitions for later use.
		\subsection{Computability of functions on the unit interval} \label{ss:compfunctions}
			We paraphrase Definition~A in Pour-El and Richards \cite[p.\ 26]{Pour-El.Richards:89}.
			\begin{definition}\label{CompDef}
				A function $f\colon [0,1] \to \mathbb R$ is called \emph{computable} if
				\bi
					\item[(a)] $f(q)$ is a computable real uniformly in a rational $q \in [0,1]$, and
					\item[(b)] $f$ is \emph{effectively uniformly continuous}:
						there is a computable $h\colon \NN \ria \NN$ such that $|x-y| < \tp{-h(n)}$ implies $|f(x) - f(y) | < \tp{-n}$ for each $n$.
				\ei
			\end{definition}
			The definition can be extended almost verbatim to functions $f\colon [0,1]^n \to \mathbb R$; in (a) we take $n$-tuples of rationals.
			Also, it suffices to consider dyadic rationals in (a).

			Every Lipschitz function $f$ is effectively uniformly continuous. Thus (a) is sufficient for the computability of $f$.
			If $n=1$ and $f$ is a continuous monotonic function, then (a) is also sufficient by \cite[Prop.\ 2.2]{Brattka.Miller.ea:nd}.

		\subsection{Differentiability}
			We use notation from \cite{Brattka.Miller.ea:nd}. For a function
			$f\colon [0,1] \to \mathbb R$, the \emph{slope} at a pair $x,y$ of distinct reals is
			\[
				S_f(x,y) = \frac{f(y)-f(x)}{y-x}.
			\]
			Note that both $x<y$ and $y<x$ are allowed here.
			Recall that the upper and lower derivatives at $z$ are defined by
			\begin{eqnarray*}
				\ol D f(z) & = & \limsup_{h\ria 0} S_f(z, z+h), \quad\text{and}\\
				\ul D f(z) & = & \liminf_{h\ria 0} S_f(z, z+h).
			\end{eqnarray*}
		\subsection{Martingales and measures}\label{ss:martingales}

			\begin{definition} \label{df:MG} {\rm A \emph{martingale} is a function $2^{< \omega} \ria \R^+_0$ such that the fairness condition
				$ M(\sss 0) + M(\sss 1) = 2 M(\sss)$
				holds for each string $\sss$. We say that
				$M$ \emph{succeeds} on a sequence of bits $Z$ if $M(Z\uhr n)$ is unbounded, where $Z\uhr n$ as usual denotes the length $n$ prefix of $Z$.
				A martingale $M \colon \, 2^{< \omega} \ria \R^+_0$ is called \emph{computable} if $M(\sss)$ is a computable real uniformly in a string $\sss$. }
			\end{definition}
			Each martingale $M$ determines a measure on the algebra of clopen sets by
			assigning $[\sss]$ the value $ M(\sss) \tp{-\sssl}$. Via Carath\'eodory's extension theorem this can be extended to the Borel sets in Cantor space.
			Measures on Cantor space correspond to measures on $[0,1]$ as long as there are no atoms on dyadic rationals.
			The measure on $[0, 1]$ corresponding to $M$ is denoted by~$\mu_M$.
			Thus, $\mu_M$ is determined by the condition \bc $\mu_M[0.\sss, 0.\sss + \tp{-\sssl}) = M(\sss) \tp{-\sssl}$.
			\ec Given a martingale~$M$, let $\Fcn (M)$ be the cumulative distribution function of the associated measure. That is,
			\[
				\Fcn(M)(x) = \mu_M[0,x).
			\]
			Then $\Fcn(M)$ is non-decreasing and left-continuous. Hence it is determined by its values on the rationals.

			\begin{lemma}\label{lem:MartLipschitz}
				Let $f = \Fcn(B)$ for a martingale $B$. Suppose that $0\le c < d$ are constants such that $B(\sss) \in [c,d]$ for each string~$\sss$.
				Then for each pair of reals $x,y$ such that $0 \le x < y \le 1$ we have
				\bc $c(y-x) \le f(y)- f(x) \le d(y-x)$. \ec
				In particular, $f$ is Lipschitz with constant $d$.
			\end{lemma}
			\begin{proof}
				To verify  the second inequality, for  an arbitrary $n$, let $(i,j) \in \NN\times\NN$ be given by
				\bc
					$(i, j) = (\lfloor x\cdot 2^n\rfloor, \lceil y\cdot 2^n\rceil)$.
				\ec
				Thus
				\begin{eqnarray*}
					&i\cdot \tp{-n}    &\le x < (i+1)\tp{-n} \quad\text{and}\\
					&(j-1)\cdot\tp{-n} &< y \le j \cdot\tp{-n}
				\end{eqnarray*}
				Then
				\begin{eqnarray*}
					f(y) - f(x) &=& \mu_M[x,y) \\
					& \le& \sum_{r=i}^{j-1} \mu_M [r \tp{-n}, (r+1)\tp{-n}) \\
					& \le& d((j-1)-i) \tp{-n} = d((j-1)\tp{-n}-(i+1)\tp{-n} + \tp{-n})\\
					& <  & d(y-x + \tp{-n}).
				\end{eqnarray*}
				The first inequality is proved in a similar way.
			\end{proof}

		\subsection{Dyadic cubes} \label{ss:dyadic_cubes}
			Let $\+Q$ be the subset of $[0,1]^n$ consisting of the vectors with a dyadic rational component.
			By a \emph{dyadic cube} we mean a closed subset $C$ of $[0,1]^n$ which for some $k$ is a product of $n$ intervals of the form $[i \tp{-k}, (i+1)\tp{-k}]$.

			Note that the binary expansion of reals yields a measure preserving map from $ [0,1]^n \setminus \+ Q$ to $(\cantor)^n $ with the product measure.
			A dyadic cube with edges of length $\tp{-k}$ corresponds to a clopen subset of the form
			$[\sss_1 ] \times \ldots \times [\sss_n]$ in $(\cantor)^n $, where each $\sss_i$ has length $k$.

			We say that $G \sub [0,1]^n$ is $\SI 1$ if $G$ is an effective union of open cubes with rational coordinates.
			By transfering a well known basic fact in Cantor space, this shows that
			from each $\SI 1$ set $V\sub [0,1]^n$ we can effectively determine a sequence $(C_i) \sN i$ of dyadic cubes that are disjoint outside $\+ Q$,
			so that $V \setminus \+Q $ equals their union outside $\+ Q$.
			We let $V_{t} = \bigcup_{i\le t} C_i$ and say that $C_i$ is enumerated into $V$ at stage $t$.

			Via the usual isometry $(\cantor)^n \cong \cantor$, we may define the binary expansion of a tuple $z =(z_0, \ldots, z_{n-1}) \in [0,1]^n \setminus \+ Q$:
			this is the bit sequence $Z$ given by $Z(ni +k) = Z_k(i)$, where $i,k \in \NN$, $k< n $, and $Z_k$ is the binary expansion of the real~$z_k$.
		\subsection{$L_p$-computability}\label{ss:Lp_comp}
			Recall that for $p \ge 1$, by $L_p([0,1]^n)$ one denotes the set of integrable functions $g\colon [0,1]^n\ria \R$ such that
			$\vectornorm g _p = ( \int |g|^p d\leb)^{1/p} < \infty$.
			In the following let $p \ge 1$ be a computable real.
			Pour-El and Richards \cite[p.\ 84]{Pour-El.Richards:89} define $g$ to be \emph{$L_p$-computable}
			if from a rational $\eps >0$ one can determine a computable function $h$ on $[0,1]^n$ such that $\vectornorm {g- h }_p < \eps$.
			Here the notion of computability for $h$ is the usual one of Subsection~\ref{ss:compfunctions}; in particular, $h$ is continuous.
			By \cite[Cor.\ 1a on p.\ 86]{Pour-El.Richards:89}
			the polynomials in $n$ variables with rational coefficients are effectively dense with respect to $\vectornorm{\cdot}_\infty$,
			so we might as well assume that $h$ is such a polynomial.

			The following is well-known in principle.
			\begin{fact}\label{fa:1G}
				If $V\sub[0,1]^n $ is $\SI 1$ and $\leb V$ is a computable real, then the characteristic function $1_V$ is $L_p$-computable,
				uniformly in a presentation of $V$ and Cauchy names for $\leb V$ and $p$.
			\end{fact}
			\begin{proof}
				Given rational $\eps > 0$, compute $t$ such that $\leb (V-V_t) < (\eps/2)^p$.
				Since $V_t$ is effectively given as a finite union of dyadic cubes, we can determine a computable function $h$ such that $\vectornorm {1_{V_t} - h}_p < \eps/2$.
				(For instance, let $h(x) = \max (0, 1- Nd(x, V_t))$, where $d$ denotes Euclidean distance, and $N\in \NN$ is an appropriate large number computed from $\eps$ and $p$.)
				This implies $\vectornorm {1_V - h}_p < \eps$.
			\end{proof}
				For $1 \le q\le p$, then $\|g\|_q \le \|g\|_p$, so every $L_p$-computable function is $L_q$-computable.
			The following  fact  and its proof suggested by a referee  shows that  for bounded functions, the converse holds. It is sufficient to consider $q=1$.
				\begin{fact}[due to the referee]\label{referee}
				If $g$ is $L_1$-computable and bounded,
				then $g$ is $L_p$-computable for each computable real $p\ge 1$.
			\end{fact}
			\begin{proof}
				Suppose $|g(x)|\le C$ for a constant $C\ge 1$.
				Uniformly in   a positive rational $\epsilon\le 1$, we can determine a computable function $h$ on $[0,1]^n$ such that
				$\|g-h\|_1<{(2C)}^{1-p}\epsilon^p$. We can assume that $|h(x)|\le C$ for each $x$, and so $\alpha:=|\frac{g(x)-h(x)}{2C}|\le 1$ for each $x$, and so
				$\alpha^p\le\alpha$. Then
				\begin{eqnarray*}
					{(2C)}^{-p} \|g-h\|^p_p
					&=& \int {\left| \frac{g-h}{2C}\right|}^p d\lambda \le 
					   \int {\left| \frac{g-h}{2C}\right|} d\lambda\\
					&=& {(2C)}^{-1}\|g-h\|_1 < {(2C)}^{-p}\epsilon^{-p}
				\end{eqnarray*}
				Thus, $\|g-h\|_p<\epsilon$.
			\end{proof}

		\subsection{The $p$-variation norm} \label{ss:variation_norm}
			The functions $f\colon [0,1]\to \RR$ of bounded variation form a Banach space under the variation norm defined by \bc $\vectornorm f _{V} = |f(0)| +V(f, [0,1])$.
			\ec We have $\vectornorm{f}_{V} \ge |\!|f|\!|_{\infty}$ (the usual sup norm).
			Let $AC_0[0,1]$ be the vector space of absolutely continuous functions $f\colon [0,1]\to \RR$ such that $f(0) =0$.
			Let $\mathcal L_1[0,1]$ denote the usual set of equivalence classes of functions in $L_1[0,1]$ modulo  equality almost everywhere.	The map $g \mapsto \lambda x. \int_0^x g$ is a computable Banach space isometry
			\bc $(\mathcal L_1[0,1], \vectornorm{\cdot }_1 ) \to (AC_0[0,1], \vectornorm{\cdot}_{V})$. \ec
			Its inverse is the derivative, which is a.e.\ defined for an absolutely continuous function. See, e.g., \cite[p.\ 376]{Carothers:00} for more detail.
			Note that the inverse is automatically computable.

			Let $ p > 1$. For a function $g\colon [0,1] \ria \R$, and $0 \le x< y \le 1$ the \emph{$p$-variation} of $g$ in $[x,y]$ is
			\[
				V_p(g,[x,y]) = \sup \left\{\sum_{i=1}^{n-1} \frac{\bigl| g(t_{i+1}) - g(t_i)\bigr|^p}{|t_{i+1}- t_i|^{p-1}} : x \le t_1 < t_2 < \ldots < t_n \le y\right\}.
			\]
			Let
			\bc $\vectornorm f _{V_p} = |f(0)| +(V_p(f, [0,1]))^{1/p}$, \ec
			and let $A_p[0,1]$ denote the class of functions $f$ defined on $[0,1]$ with $f(0) = 0$ and $\vectornorm f _{V_p} < \infty$.
			Riesz~\cite{Riesz:1910} showed that each function in $A_p[0,1]$ is absolutely continuous.
			In analogy to the isometry of Banach spaces above, he also showed that the map $g \to \lambda x. \int_0^x g$ yields an isometry

			\bc $(\mathcal L_p[0,1], \vectornorm{\cdot }_p ) \to (A_p[0,1], \vectornorm{\cdot}_{V_p})$, \ec
			with inverse the derivative. Note that for computable $p$, this map is computable with respect to the relevant norms.

			We remark that for computable $p \ge 1$, the space $(\mathcal L_p[0,1], \vectornorm{\cdot }_p ) $ is also effectively isomorphic,
			in the sense of computable Banach spaces, to the Sobolev space $W^{1,p}(0,1)$.
			This uses the so-called ACL characterization of Sobolev spaces.
			See, e.g., \cite[Thm.\ 2.1.4]{Ziemer:89}.

		\subsection{Interval-r.e.\ functions}\label{ss:interval-r.e.}

			\n
			We recall that a real $x \in [0,1]$ is \emph{left-r.e.}\ if the set $\{q \in
			\mathbb Q\colon \, q < x\}$ is r.e. If this left cut equals $W_e$, we say that $e$ is an \emph{index} for $x$.
			(Such a real is also called ``left-computable'', and sometimes ``lower semicomputable'', in the literature.)

			\begin{definition}\label{def:intervalce}
				A non-decreasing function $f$ defined on $[0,1]$ is called
				\emph{interval-r.e.} if $f(0) = 0$, and $f(y)-f(x)$ is left-r.e.\ uniformly in rationals $x<y$.
			\end{definition}
			Suppose in Definition~\ref{def:intervalce} we drop the restriction on $x,y$ being rational,
			and require the stronger condition that $f(y)-f(x)$ is left-r.e.\ relative to Cauchy names of reals $x< y$.
			The variation of a computable function, and the functions $f_M$ defined below, satisfy this stronger condition. For continuous functions,
			the two conditions are equivalent. For suppose the weaker condition in Definition~\ref{def:intervalce} holds.
			If $(p_n)\sN n$ and $(q_n)\sN n$ are Cauchy names for $x$ and $y$ respectively,
			then $x\le p_n + \tp{-n}$ and $q_n- \tp{-n} \le y$ for each $n$.
			Then by continuity $f(y)- f(x)$ is the sup of the values $f(q_n- \tp{-n} ) - f(p_n + \tp{-n})$ where $p_n + \tp{-n} \le q_n- \tp{-n}$.
			This is left-r.e.\ in the Cauchy names by hypothesis.

			See~\cite[Ch.\ 2]{Nies:book} or \cite{Downey.Hirschfeldt:book}
			or \cite{LV08} for background on prefix-free machines and prefix-free complexity~$K$.
			For a set $B \sub \strcantor$ let $\Opcl B$ denote the open set $\{ X \in \cantor \colon \, \ex n \, X \uhr n \in B\}$.
			Let
			$S$ be a prefix-free machine. We identify a binary string $\gamma$
			with the dyadic rational $0.\gamma$. The following function
			is interval-r.e.:
			\[
				f_S (x) = \leb \Opcl {\{ \sss \colon \, S(\sss) < x \}}.
			\]
			Thus, $f_S(x)$ is the probability that $S$ prints a dyadic rational less than~$x$.
			Note that $f_S$ is left continuous ($f_S(x) = f_S(x^-)$ for each~$x$) and hence, being increasing, lower semi-continuous.
			Furthermore, $f_S$ is discontinuous at $x$
			(namely, $f_S(x) < f_S(x^+)$) precisely if $x<1 $ is a dyadic rational in the domain of $S$.

			Let $\mathbb U$ be a universal prefix free machine, and consider the increasing interval-r.e.\ function $f_{\UM} (x) $.
			Then $f_\UM(1) = \Om_\UM$, and thus $f_\UM$ is not computable on the rationals.

			\begin{prop}
				Let $z\in [0,1]$. If $\ol Df_{\UM} (z)< \infty$ then $z$ is \ML\ random.
			\end{prop}

			\begin{proof}
				Suppose that $z$ is not ML-random. Given $c \in \NN$ pick $n$ such that $K(z\uhr n) \le n-c$. Let $h = - \tp{-n}$. We have
				\bc $\tp{-n+c} \le \leb \Opcl {\{ \sss \colon \, \mathbb U(\sss) \in [z+h,z)\} } = f_{\UM}(z) -f_{\UM}(z+h)$. \ec
				Therefore $ 2^c \le (f_{\UM}(z+h) -f_{\UM}(z))/h $.
			\end{proof}

			\n
			In \cite{Bienvenu.Greenberg.ea:preprint} it is shown that, conversely, if $z$ is \ML\ random then each interval-r.e.\ function has finite upper derivative.
			In contrast, there is a function of the form $f_S$ for a prefix-free machine $S$ that is not differentiable at Chaitin's $\Om$.
			Simply let the domain of $S$ generate the open set in Cantor space corresponding to $[0,\Om)$ (i.e., $x < \Om $ iff $\ex n \, S(x\uhr n) \DA $).
			Then $f_S$ increases to $\Om$ in smaller and smaller steps, and it is constant equal to $\Om$ thereafter.
			It is easy to check that $f'_S(\Om)$ fails to exist.
			\cite{Bienvenu.Greenberg.ea:preprint} also show that
			a randomness property of a real $z$ slightly stronger than \ML's ensures that each interval-r.e.\ function is differentiable at~$z$.
			We will discuss this in the concluding section of the paper.

	\section{Characterizing the variation of computable (Lipschitz) functions}
		\label{s:LipVar}

		\n
		Let $g\colon [0,1] \ria \R$ be a computable function.
		Since $V(g,[x,y]) + V(g, [y,z])$ $= V (g, [x,z])$ for $x< y< z$, we see that the function $f(x) = V(g, [0,x])$ is interval-r.e.\
		and continuous. Note that if $f$ is Lipschitz with constant $c$, then the function $g$ is necessarily Lipschitz with constant at most $c$, because for $x<y$ we have
		\bc $|g(y ) - g(x) | \le V(g, [x,y]) = V(g, [0,y]) - V(g, [0,x])$. \ec
		As our main result in this section, we will prove the converse for Lipschitz functions:
		every interval-r.e., non-decreasing Lipschitz function $f$ is of the form $ V(g, [0,x])$ for some computable Lipschitz function $g$.
		We begin with the simpler result that the total variation can be any given left-r.e.\ real.

		\begin{fact}\label{fact:easiest}
			For each left-r.e.\ real $\aaa$, $0 \le \aaa \le 1$, there is a computable function $g$ which is Lipschitz with constant $1$ such that $V(g, [0,1]) = \aaa$.
		\end{fact}
		\begin{proof}
			For each interval of dyadic rationals $[p,q]$, where $p = i \tp{-n}$, $q = j \tp{-n}$, $0 \le i \le j < \tp n$, and each $k > n$,
			let $W(p,q; k)$ be the function that zigzags
			\[
				(q-p)2^k = (j-i) \tp{k-n}
			\]
			times within $[p,q]$, with slope $\pm 1$. Then the total variation of $W(p,q; k)$ is $q-p$.

			Now let $\aaa = \lim_s \aaa_s$ where $(\aaa_s) \sN s$ is an increasing effective sequence of dyadic rationals $\aaa_s $ of the form $ i \tp{-n}$ with $n< s$. Let

			\bc $g = \sum_s W(\aaa_s, \aaa_{s+1}; s+1)$. \ec
			It is easy to check that $g$ is a computable function that is Lipschitz with constant~$1$.
			Furthermore, since variation is additive over partitions into disjoint intervals, $g$ has variation~$\aaa$.
			(For intuition, note that as $\alpha_s$ approaches $\alpha$, the oscillations become flatter and flatter. To the right of $\alpha$, $g$ is constant.)
		\end{proof}
		\n
		It is well known that for  any $L_1$-computable function $h$, the function $x \mapsto \int_0^x h$ is computable.  		Using the foregoing fact we provide an example showing that  the converse fails.
		\begin{cor}
			Some nondecreasing computable Lipschitz function $u$ is not of the form $x \mapsto \int_0^x v$ for any $L_1$-computable function $v$.
		\end{cor}

		\begin{proof}
			Let $\aaa$ be a left-r.e.\ noncomputable real, and let $g$ be as in Fact~\ref{fact:easiest}. Let $u(x)= x -g(x)$.
			Then $u$ is nondecreasing and Lipschitz with constant $2$.
			Assume for a contradiction that $u$ is of the form $\int_0^x v$ for an $L_1$-computable function~$v$.

			We have $g(x) = \int_0^x h$, where $h= 1-v$. Then $V ( g, [0,x]) = \int_0^x |h|$ by a classic result of analysis
			(see \cite[Prop.\ 5.3.7]{Bogachev.vol1:07}). Furthermore, $|h|$ is $L_1$-computable. Thus, $V ( g, [0,1]) $ is a computable real, a contradiction.
		\end{proof}

		\n
		The proof of our main result in this section makes use of signed martingales,
		namely, functions $L \colon \strcantor \to \mathbb R$ satisfying the martingale equality $2 L(\sss) = L(\sss0 ) + L(\sss 1)$ for each string $\sss$.
		Given a signed martingale $L$, let
		\[
			V_L(\sigma) = \sup_k \tp{-k} \sum_{|\eta| =k} |L(\sss \eta)|.
		\]
		It is easy to build a computable $L$ such that $V_L(\emptyset) = \infty$. If $V_L(\sss) < \infty $ for each $\sss$, we say that
		$V_L$ is the \emph{variation martingale} of $L$.
		Note that the expression on the right is nondecreasing in $k$.
		Thus, if $L$ is computable then the variation martingale $V_L$ is a left-r.e.\ (non-negative) martingale.

		We say that a martingale $M \colon \strcantor \to \RR^+_0$ is \emph{non-atomic} if
		the corresponding measure $\mu$ on Cantor space is non-atomic (see the discussion after Definition \ref{df:MG}).
		This means that for each $X \in \cantor$ we have $M(X\uhr n) = o(2^n)$.
		By compactness of Cantor space, the function $X \mapsto \mu[0,X]$ is uniformly continuous.
		So, in fact we have the apparently stronger condition that $M(\sss)= o(\tp{\sssl})$ for each string~$\sss$.

		\begin{lemma}\label{lem:left-r.e. and variation}
			Let $M\colon \strcantor \to \mathbb R^+_0$ be a left-r.e.\ non-atomic martingale.
			Then there is a computable signed rational-valued martingale $L$ such that $M= V_L$.
			Furthermore, $|L(\sss)| \leq M(\sss)$ for each $\sss$.
		\end{lemma}
		\begin{proof}
			Since $M$ is left-r.e., we may assume that $M(\sss)= \sup_s M_s(\sss)$ where each $M_s$ is a recursive martingale uniformly in $s$,
			sending strings to rational numbers, and $M_s(\sss) \leq M_t(\sss)$ whenever $s \leq t$.
			For natural numbers $a<b$ and  a string $\sigma$ of length $a$, define the approximation of $V_L$ on level $b$ by
			\[
				V_{L,b}(\sigma) = 2^{-(b-a)} \sum_{| \eta | = b-a} | L(\sigma \eta) |.
			\]
			Now one defines the new martingale $L$ inductively with $L(\varnothing) = 0$ where $\varnothing$ is the empty string.
			At stage $s$, the idea is to define $L$ for strings of lengths
			\[
				\ell_s+1, \ell_s+2, \ldots, \ell_{s+1}
			\]
			where
			$\ell_{s+1}$ will be chosen so that for all
			$\sss \in \{0,1\}^{\ell_s}$, the difference between $M_s(\sss)$ and $V_{L,\ell_{s+1}}(\sss)$
			is less than $2^{-s}$.

			One defines inductively for all strings of length $\ell_s, \ell_s+1, \ldots$
			the value of $L(\sss 0)$ and $L(\sss 1)$, using that
			$|L(\sss)| \leq M_s(\sss)$ and imposing the same on $L(\sss 0)$ and $L(\sss 1)$.
			Choose $a\in \{0,1\}$ such that $M_s(\sss a) \leq M_s(\sss (1-a))$ and
			define $L$ on $\sss 0$ and $\sss 1$ as follows:
			\begin{eqnarray*}
				L(\sss a) & = & \begin{cases}
				M_s(\sss a) & \text{if } L(\sss) \geq 0; \cr
				-M_s(\sss a) & \text{if } L(\sss) < 0; \cr \end{cases} \\
				L(\sss (1-a)) & = & \begin{cases}
				M_s(\sss a)+2 \cdot (L(\sss)-M_s(\sss a))
				& \text{if } L(\sss) \geq 0; \cr
				-M_s(\sss a)+2 \cdot (L(\sss)+M_s(\sss a))
				& \text{if } L(\sss) < 0. \cr \end{cases}
			\end{eqnarray*}
			Note that $L$ satisfies the martingale equality.
			If $L(\sss) \geq 0$ then
			\bc $L(\sss (1-a)) \leq M_s(\sss a)+2 \cdot (M_s(\sss)-M_s(\sss a)) =
			M_s(\sss (1-a))$ \ec
			and
			\bc $L(\sss (1-a)) \geq - M_s(\sss a) \geq -M_s(\sss (1-a))$. \ec
			Hence $|L(\sss a)| = M_s(\sss a) \leq M(\sss a)$
			and $|L(\sss (1-a))| \leq M_s(\sss (1-a)) \leq M(\sss (1-a))$ in this case.
			If $L(\sss) < 0$ then
			\bc $L(\sss (1-a)) \geq -M_s(\sss a)+2(-M_s(\sss) + M_s(\sss a)) =
			-M_s(\sss (1-a))$ \ec
			and
			\bc $L(\sss (1-a)) \leq M_s(\sss a) \leq M_s(\sss (1-a))$. \ec
			Again this implies $|L(\sss a)| = M_s(\sss a) \leq M(\sss a)$
			and $|L(\sss (1-a))| \leq M_s(\sss (1-a)) \leq M(\sss (1-a))$.

			Furthermore, note that whenever $L(\sss) = M_s(\sss)$ then
			$L(\sss 0) = M_s(\sss 0)$ and $L(\sss 1) = M_s(\sss 1)$;
			whenever $L(\sss) = -M_s(\sss)$ then
			$L(\sss 0) = -M_s(\sss 0)$ and $L(\sss 1) = -M_s(\sss 1)$.

			So one can show by induction that there are on each of the
			levels $\ell_s+1, \ell_s+2, \ldots$
			at most $2^{\ell_s}$ strings $\sss$ with
			$|L(\sss)| \neq M_s(\sss)$. As $M(\sss) = o(\tp{\sssl})$,
			there is some level $\ell_{s+1}$ such that for all $\sss \in \{0,1\}^{\ell_s}$
			the difference between $V_{L,\ell_{s+1}}(\sss)$ and $M_s(\sss)$ is at
			most $2^{-s}$. From this fact, one can conclude that for each string
			$\sss$, the difference between $M(\sss)$ and $V_{L,\ell_t}(\sss)$
			is for $\ell_t \geq |\sss|$ bounded by $M(\sss)-M_t(\sss)+2^{-t}$.
			So $V_L(\sss) = M(\sss)$ for all strings $\sss$.

			The condition $|L(\sss)| \leq M(\sss)$ can
			be verified by an easy induction using that $M_s(\sss) \leq M_{s+1}(\sss) \leq M(\sss)$
			for all~$s$.
		\end{proof}

		\begin{theorem} \label{thm:left-r.e._Lipschitz_variation}
			Let $f$ be a non-decreasing interval-r.e.\ function with Lipschitz constant $c$.
			Then there is a computable function $g$ with the same Lipschitz constant $c$ such that $f(x) = V(g, [0,x])$ for each $x\in [0,1]$.
		\end{theorem}

		\begin{proof}
			We define a left-r.e.\ martingale $M$ by $M(\sss) = S_f(0.\sss, 0.\sss+ \tp{-\sssl})$. Note that $M$ is bounded by the Lipschitz constant $c$ for $f$.
			Let $L$ be the signed computable rational-valued martingale with $V_L = M$ obtained through Lemma~\ref{lem:left-r.e. and variation}.
			Then $|L|$ is bounded by $c$ as well.

			For a dyadic rational $0.\sss$, $\sss$ a binary string, we let
			\[
				g(0.\sss)= \tp{-\sssl} \sum \{ L(\tau)\colon \, 0.\tau < 0.\sss \lland |\tau| = \sssl\}.
			\]
			Note that by the martingale equality, $g$ is well defined on the dyadic rationals in $[0,1)$. Clearly, for strings $\sss, \rho$ of the same length $n$, we have
			\[
				|g(0.\sss) - g(0.\rho)|
				\le \tp{-n} \sum \{ |L(\tau)|\colon \, 0.\rho \le 0.\tau < 0.\sss \lland |\tau| = n\}
				\le c| 0. \sss - 0. \rho |.
			\]
			Thus $g$ is Lipschitz on the dyadic rationals.
			Therefore by the remark after Definition~\ref{CompDef}, $g$ can be extended to a computable function on $[0,1]$ with Lipschitz constant $c$, also denoted $g$.

			We claim that $f(x) = V(g, [0,x])$ for each $x \in [0,1]$. By continuity of $f$ and $g$, we may assume
			that $x =0.\sss$ for string $\sss$ of length $n$, and that
			in the definition of $V(g, [0,x])$ we only consider partitions consisting of all the dyadic rationals $0.\rho < 0.\sss$,
			where all $\rho$ have the same length $k \ge \sssl$. Then
			\[
				V(g, [0,x]) = \tp{-n} \sum \{V_L(\tau) \colon \, 0.\tau < 0. \sss \lland |\tau| = n.\}
			\]
			Since $f(0)= 0$ we have
			\[
				f(x) = \tp{-n} \sum \{S_f(0.\tau, 0.\tau + \tp{-n}) \colon \, 0.\tau < 0. \sss \lland |\tau| = n \}.
			\]
			Since $M(\tau ) = S_f(0.\tau, 0.\tau + \tp{-n})$ and $M=V_L$, this establishes the claim.
		\end{proof}

		\n {\bf Extension to all continuous interval-r.e.\ functions.}
		After seeing our result, Jason Rute has extended the technique of
		Theorem~\ref{thm:left-r.e._Lipschitz_variation}, discarding the hypothesis that the given function be Lipschitz:

		\begin{theorem}[with Rute]\label{thm:Rute-extension}
			Let $f$ be a continuous non-decreasing interval-r.e.\ function.
			Then there is a computable function $g$ such that $f(x) = V(g, [0,x])$ for each $x\in [0,1]$.
		\end{theorem}
		\begin{proof}
			Suppose that $M$ is the left-r.e.~martingale associated with $f$. It is not enough to just do what we did before.
			The problem is that given a signed martingale $L$, the function $g$
			is not necessarily computable. (Here, although $L$ can be negative, it is correct to write $g=\mathsf{cdf}(L)$.) This can be fixed by being a bit more
			careful about which $L$ we construct in
			Lemma~\ref{lem:left-r.e. and variation}.

			To construct $L$ from $M$, we follow the proof of
			Lemma~\ref{lem:left-r.e. and variation}, with one adjustment.
			Note that for each $M_{s}$ in the proof of Lemma~\ref{lem:left-r.e. and variation}, there is some stage $k_{s}$
			in which $2^{-|\sigma|}\cdot M_{s}(\sigma)\leq2^{-s}$ for all $|\sigma|\geq k_{s}$.
			
			Indeed, $M_{s}$ has no atoms, hence for each $x$, $2^{-k}\cdot M_{s}(x\uhr k)\searrow0$
			as $k\rightarrow\infty$. In particular, for each $x$ there is a $k$ such that
			$2^{-k}\cdot M_{s}(x\uhr k)\leq 2^{-s}$.
			By compactness, there is in fact a single  $k$ such that for all $x$,
			$2^{-k}\cdot M_{s}(x\uhr k)\leq 2^{-s}$.

			Do not switch from $M_{s-1}$
			to $M_{s}$ in the construction until after stage $k_{s+1}$ (we can
			assume $M_{0}=M_{1}=0$).

			Let $j_{s}$ be the stage at which we switch to $M_{s}$. Clearly, $j_{s}\geq k_{s+1}$.
			So for any $\sigma$ such that $j_{s+1}>|\sigma|\geq j_{s}$ we have
			$|L(\sigma)|\leq M_{s}(\sigma)$ by construction. Therefore,
			\[
				2^{-|\sigma|}\cdot|L(\sigma)|\leq2^{-|\sigma|}\cdot M_{s}(\sigma)\leq2^{-|\sigma|}\cdot M_{s+1}(\sigma)\leq2^{-(s+1)}.
			\]
			Let $g=\mathsf{cdf}(L)$.
			By the same proof as in Theorem \ref{thm:left-r.e._Lipschitz_variation}, $f(x)=V(g,[0,x])$.

			It remains to show that $g$ is computable. Clearly, $g(0.\sigma)$ is uniformly computable for all $\sigma$. Let $\nu$
			be the signed measure associated with $L$, so that
			\[
				\nu([0.\sigma, 0.\sigma + 2^{-|\sigma|})) = 2^{-|\sigma|} L(\sigma)\quad\text{and}\quad\mathsf{cdf}(\nu)=\mathsf{cdf}(L).
			\]
			Also, for each $s$, let $\mu_{s}$ be the measure associated with
			$M_{s}$. The above formula then becomes (writing $\nu(\sigma)$ for $\nu([0.\sigma + 2^{-|\sigma|})$)
			\[
				|\nu(\sigma)|\leq\mu_{s}(\sigma)\leq\mu_{s+1}(\sigma)\leq2^{-(s+1)}
			\]
			for $j_{s+1}>|\sigma|\geq j_{s}$. Pick $a\in[0,1]$.
			To compute $g(a)$ within $2^{-(s-1)}$ uniformly from $a$, let $\sigma = a \upharpoonright j_s$ (i.e., $a\in [\sigma]$ and $|\sigma|=j_s$). From the
			Cauchy name for $a$, one can determine one of the values $\{ g(0.\sigma),
			g(0.\sigma + 2^{-j_s}) \}$. Notice that
			\[
				|g(0.\sigma + 2^{-j_s}) - g(0.\sigma)| = |\nu(\sigma)| \leq 2^{-(s+1)}.
			\]
			We claim that $|g(a) - g(0.\sigma)| \leq 2^{-s}$.

			The claim is equivalent to saying that $|\nu[0.\sigma,a)|\leq2^{-s}$.
			We know that $[0.\sigma,a)\subseteq[\sigma]$. Now, break up $[0.\sigma,a)$
			into $[\tau_{n}]\cup\cdots\cup[\tau_{1}]\cup[0.\tau_{0},a)$, where
			$\tau_{n},\ldots,\tau_{1},\tau_{0}$ are adjacent, $|\tau_{i}|=j_{s+1}$
			for $1\leq i\leq n$, and $\tau_{0}=a\upharpoonright j_{s+1}$. Since
			$|\sigma|=j_{s}$ and $|\tau_{i}|=j_{s+1}$, we have
			\begin{align*}
				|\nu[0.\sigma,a)| & \leq|\nu(\tau_{n})|+\cdots+|\nu(\tau_{1})|+|\nu[0.\tau_{0},a)|\\
				& \leq\mu_{s+1}(\tau_{n})+ \cdots + \mu_{s+1}(\tau_{1})+|\nu[0.\tau_{0},a)|\\
				& \leq\mu_{s+1}(\sigma)+|\nu[0.\tau_{0},a)|\\
				& \leq2^{-(s+1)}+|\nu[0.\tau_{0},a)|.
			\end{align*}
			Continuing by recursion, we have for each $m$ that
			\[
				|\nu[0.\sigma,a)|\leq |\nu[0.\tau_{0}^{m},a)| + \sum_{i=1}^{m} 2^{-(s+i)}
			\]
			where $\tau^m_0 = a \upharpoonright j_{s+m}$. Let $\mu$ be the measure associated
			with $M$. Then
			\[
				|\nu[0.\tau_{0}^{m},a)|\leq\mu[0.\tau_{0}^{m},a)\rightarrow0\quad\text{as}\quad m \rightarrow \infty,
			\]
			and hence $|\nu[0.\sigma,a)|\leq\sum_{i=1}^{\infty}2^{-(s+i)}=2^{-s}$.
		\end{proof}

	\section{Computable randomness and Lipschitz functions}
		\subsection{Characterizing computable randomness}
			\n
			Schnorr \cite{Schnorr:75} introduced the following notion.
			\begin{definition}\label{df:CR}
				{\rm A sequence of bits $Z$ is called \emph{computably random} if no computable martingale succeeds on $Z$.
				A real $z \in [0,1]$ is called \emph{computably random} if  a   binary expansion of $z$  is computably random.}
			\end{definition}
			Note that we can ignore the case that there are two binary expansions of $z$, because in that case $z$  is a dyadic rational and hence computable. 
			Also, it   suffices to require that no rational-valued computable martingale succeeds on a binary expansion of $z$
			by a result of \cite{Schnorr:75} (for a recent reference see \cite[7.3.8]{Nies:book}).
			Here computability of the martingale can even be taken in the strong sense that we can uniformly compute a canonical index for the rational value (rather than
			give a sequence of approximations that happen to converge to a rational value).
			For more background on computable randomness see \cite[Ch.~7]{Nies:book} or \cite{Downey.Hirschfeldt:book}.

			We now characterize computable randomness by differentiability of computable Lipschitz functions,
			analogously to Theorem~\ref{thm:Brattka} due to \cite{Brattka.Miller.ea:nd} mentioned in the introduction.

				\begin{theorem} \label{thm:comprd_Lipschitz}
				Let $z\in [0,1]$. Then
				$z$ is computably random $\LR$ 
				
				\hfill
each computable Lipschitz function $f \colon \, [0,1] \ria \mathbb R$ is differentiable at~$z$.
			\end{theorem}
			\begin{proof}[Proof of Theorem \ref{thm:comprd_Lipschitz}]
				\rapf Suppose $f$ is a computable Lipschitz function with a Lipschitz bound $c \in \NN$. As observed above, the function $g(x)= f(x) + cx$ is nondecreasing.
				Since $g$ is computable, by \cite[Thm.\ 4.1]{Brattka.Miller.ea:nd}, $g'(z)$, and hence $f'(z)$, exists.

				\vsps

				\n \lapf  We may assume that~$z$ is not a dyadic rational. Suppose $z$ is not computably random, so  some computable martingale $M$ succeeds on the binary expansion $Z$ of  $z$.
			We will build a computable Lipschitz function~$f$ that is not differentiable at $z$.

			We let  $f = \mathsf{cdf} ( B)$ as defined in Subsection~\ref{ss:martingales}, for   a computable bounded martingale $B$
			that oscillates between     values  $ \ge 3$  and values $\le 2$ when processing longer and longer initial segments of the binary expansion of $z$.
			We build $B$ from $M$.
			Note that for any string $\sss$, the value $B(\sss)$ is the slope of $f$ between the dyadic rationals $0. \sss$ and $0. \sss + \tp{-\sssl}$,
			so $f$ is not differentiable at $z$.
			The argument in the usual proof of the Doob martingale convergence theorem
			(see, e.g., \cite{Durrett:96}) turns oscillation of a martingale into success of another martingale.
			In a sense, we reverse this argument, turning the success of $M$  into oscillation of $B$.
			
					As mentioned above, we may assume that $M$ only takes positive rational values which can be computed in a single output from the input string.     We may also assume that $M$ has the savings property 
					\bc $M (\sss \eta) \ge M(\sss ) -1$ for each strings $\sss, \eta$; \ec 
					see e.g.\ Exercise~\cite[7.1.14]{Nies:book} and its solution, or \cite{Downey.Hirschfeldt:book}.
					At each $\sss$, the martingale $B$ is in one of two possible phases.
In the \emph{up phase}, it  adds the capital that $M$ risks, until its value $B(\sss)$ reaches $3$ (if this value would exceed $3$, $B$   adds   less in order to ensure the value equals $3$). In
the \emph{down phase}, $B$   subtracts the capital that $M$ risks, until the value $B(\sss)$ reaches $2$.

In more detail, the construction of $B$ is as follows.  Inductively we show that if $B$ is in the up phase at $\sss$, then  $B(\sss)<3$, and if $B$ is in the down phase at $\sss$ then $B(\sss) > 2$.  
At the empty string $\emptyset$, the martingale $B$ is in the up phase and $B(\emptyset) =2$.   Thus the inductive condition holds at the empty string. 
Suppose now that  $B(\sss)$ has been defined.

\vsp

\n {\em Case 1: $B$ is in the up phase at $\sss$.}   Let  $$r_k = B(\sss) + M(\sss k) - M(\sss). $$
\n
 If $r_0, r_1 < 3$ then let $B(\sss k) = r_k$; stay in the up phase at both $\sss 0$ and $\sss 1$.
Otherwise, since $M$ is a martingale and $B(\sss)< 3$, there is a unique $k$ such
that $r_k \ge 3$. Let $B(\sss k) = 3$ and $B(\sss (1-k)) = 2B(\sss) -3$. 
Go into the down phase at $\sss k$, but stay in the up phase at $\sss (1-k)$. Note that the inductive condition is maintained at both $\sss 0$ and $\sss 1$.  

\vsp 

\n {\em Case 2: $B(\sss)$ is in the down phase.}  Let  $$r_k = B(\sss) -  (M(\sss k) - M(\sss)). $$
\n
 If $r_0, r_1  > 2 $ then let $B(\sss k) = r_k$; stay in the up phase at both $\sss 0$ and $\sss 1$.
Otherwise, since $M$ is a martingale and $B(\sss)>2$, there is a unique $k$ such
that $r_k \le 2$. Let $B(\sss k) = 2$ and $B(\sss (1-k)) = 2B(\sss) -2$. 
Go into the up  phase at $\sss k$, but stay in the down  phase at $\sss (1-k)$. The inductive condition is maintained at both $\sss 0$ and $\sss 1$.  

\begin{claim} For each string $\tau$ we have  $1 \le B(\tau) \le 4$. \end{claim}
To see this, suppose first that $B$ is in  the up phase at $\tau$. Then $B(\tau) \le 3$. For the lower bound on $B(\tau)$, suppose  that $B$ entered the up phase at the  string $\sss \preceq \tau$ with  $\sssl$  maximal.  Then $B(\sss) =2$. By the savings property we have $M(\tau ) - M(\sss) \ge -1$. Therefore $B(\tau)  = B(\sss) + M(\tau) - M(\sss) \ge 1$.

Next suppose   that $B$ is in  the down phase at $\tau$. Then $B(\tau) \ge 2$. For the upper bound on $B(\tau)$, suppose that  $B$ entered the down  phase at the  string $\sss \preceq \tau$ with  $\sssl$ maximal.  Then $B(\sss) =3$. By the savings property we have  $B(\tau)  = B(\sss) - ( M(\tau) - M(\sss)) \le 4$.
This shows the claim.

Since $M$ succeeds on the binary expansion  $Z$ of $z$, it is clear that $B$ oscillates  along $Z$  as described above.  
\end{proof}
	\section{Schnorr randomness in $[0,1]^n$ and Lipschitz functions}\label{s:Schnorr}
		\subsection{Schnorr randomness} \label{ss:Schnorr}
			We let $\leb$ denote Lebesgue measure on $[0,1]^n$.
			For an introduction to algorithmic randomness in spaces other than Cantor space and $[0,1]$, see \cite{Hoyrup.Rojas:09}.
	We say that  $G \sub [0,1]^n$ is $\SI 1$ if $G = [0,1]^n \cap H$ where $H\subseteq\mathbb R^n$ is an effective union of open cubes with rational coordinates.
			A uniformly $\SI 1$ sequence $(G_m)\sN m$ is called {\it Schnorr test} if $\leb G_m \le \tp{-m}$ and $\leb G_m$ is a computable real uniformly in $m$.
			A point $z \in [0,1]^n$ is called \emph{Schnorr random} if $z \not \in \bigcap_m G_m$ for each Schnorr test $(G_m)\sN m$.

		\subsection{Characterizing Schnorr randomness}
			We characterize Schnorr randomness of a real by being a {weak} Lebesgue point of each bounded $L_1$-computable function.
			As a corollary, we obtain a characterization in terms of differentiability at the real of all Lipschitz functions that are computable in the $p$-variation norm for a fixed computable real $p\ge 1$.
			(See the introduction for more background and the relationship of this result to \cite{Pathak.Rojas.ea:12}.)
			\begin{theorem}\label{thm:Lp_Schnorr}
	 Let $z \in [0,1]^n$. Then
				$z$ is Schnorr random $\LR$ 
				
				\hfill 	
				$z$ is a weak Lebesgue point of each $L_1$-computable
				boun\-ded function $g\colon [0,1]^n \to \RR$.
			\end{theorem}
			\begin{proof}
				\n \rapf This is immediate by \cite[Thm.\ 1.6]{Pathak.Rojas.ea:12}.
				They show that in fact, $z$ is a weak Lebesgue point of each $L_1$-computable function, bounded or not.

				\n \lapf This implication is proved by contraposition 
				Suppose $z \in [0,1]^n$ is not Schnorr random.
				We   build a bounded $L_1$-computable  function $g\colon [0,1]^n \to \RR$,
				and a sequence of dyadic cubes $C_m \downarrow z$, such that
				\begin{equation} \label{eqn:ours_is_better}
					\limsup_m \frac{\int_{C_m} g}{ \leb C_m } =1 \quad\text{and}\quad\liminf_m \frac{\int_{C_m} g}{ \leb C_m } = -1.
				\end{equation}
			Then  $z$ is not a weak Lebesgue point of  $g$.

				Recall that $\+Q$ denotes the subset of $[0,1]^n$ consisting of the vectors with a dyadic rational component.
				Clearly  $g$ exists for $z \in \+Q$, so we may assume that $z \not \in \+ Q$.
				In the following all assertions of inclusion relations and disjointness for subsets of $[0,1]^n$ are meant to hold only on $[0,1]^n \setminus \+ Q$.

				Let $(V_m)\sN m$ be a Schnorr test in $[0,1]^n$ such that $z \in \bigcap_m V_m$. We will modify $(V_m)\sN m$ to obtain a new Schnorr test
				$(G_m)\sN m$ with $\bigcap_m V_m \sub \bigcap_m G_m $,
				and in particular $z \in \bigcap_m G_m$ (using that $z \not \in \+ Q$).
				Thereafter we will show that the bounded function $g$ defined by $g(x) = 0$ for $x \in \bigcap_m G_m$ and
				\begin{equation}
					\label{eqn:g_function_def}
					g(x)= \sum_{m=0}^\infty (-1)^m 1_{G_m}(x)
				\end{equation}
				for $x \not \in \bigcap_m G_m$ is as required.

				Recall from Subsection~\ref{ss:dyadic_cubes} that at each stage $t$ we have a set
				$V_{m,t}$ which is a finite union of dyadic cubes that are disjoint (outside $\+ Q$).

				\vsps

				\n {\it Construction of the Schnorr test $(G_m)\sN m$.} Set $G_{0,s} = [0,1]^n$ for each $s$.
				Suppose inductively we have defined a computable enumeration $(G_{m,s})_{s\in\NN}$ of $G_m$.
				Suppose that a dyadic cube $C$ is enumerated into $G_{m,s}$.
				Let $r\ge s$ be least such that
				\begin{equation} \label{eqn:smaller_cubes}
					2^{-r}
					\le 2^{-m-1}\leb C.
				\end{equation}
				(this will be used to show that $z$ is not a weak Lebesgue point of $g$). Enumerate the set $V_r \cap C$ into $G_{m+1}$. In more detail, for all $t\ge s$ enumerate
				the set $V_{r,t} \cap C$ into $G_{m+1,t}$. This ends the construction.

				Note that because $C$ is disjoint from $G_{m,s-1}$, the set $V_r \cap C$ is also disjoint from $G_{m+1,s-1}$;
				this will be needed when we verify that the reals $\leb G_m$ are uniformly computable.

				\begin{claim}
					We have $\bigcap_{r\in\NN} V_r \subseteq G_m$ for each $m$.
				\end{claim}
				\n We verify the claim by induction on $m$. Clearly $\bigcap_{r\in\NN} V_r \subseteq G_0 = [0,1]^n$.
				Inductively suppose that $\bigcap_{r\in\NN} V_r \subseteq G_{m}$.
				Thus every point in the set $\bigcap_{r\in\NN} V_r \setminus \+ Q$ is in some cube $C$ enumerated into $G_{m}$.
				Then by construction we have $\bigcap_{r\in\NN} V_r \cap C \sub G_{m+1}$. This shows the claim.

				We now verify that $(G_m)\sN m$ is a Schnorr test. Clearly (\ref{eqn:smaller_cubes}) implies that $\leb G_m \le \tp{-m}$ for each $m$.
				\begin{claim}\label{cl:Gm_comp_measure}
					$\lambda(G_m)$ is a computable real uniformly in $m$.
				\end{claim}
				Note that $\lambda(G_0) = 1$.
				Inductively suppose we have a procedure to compute the real $\leb G_m$.
				Given a rational $\epsilon >0$, we will (uniformly in $m$) compute a $t$ such that $\lambda(G_{m+1} \setminus G_{m+1,t}) <2 \epsilon$.
				By the inductive hypothesis we can compute $s$ such that $\leb (G_m \setminus G_{m,s}) < \eps$.
				Let $N$ be the number of cubes in $G_{m,s}$. Denote these cubes $C_0, \ldots, C_{N-1}$.

				Since the quantities $\lambda(V_m)$ are computable uniformly in $m\in\NN$, we may compute $t \ge s$ such that for all
				$i < N$, we have
				\[
					\lambda(V_{r_i} \setminus V_{r_i,t}) < \frac\epsilon{2N},
				\]
				where $V_{r_i}$ is the set selected on behalf of the cube $C_i$ in the construction of $G_{m+1}$.
				By construction, for each $t \ge s$ we have
				\[
					\textstyle
					G_{m+1,t} \cap G_{m,s} = \bigcup_{i=0}^{N-1} (V_{r_i,t} \cap C_i ).
				\]
				Then, since $G_{m,s}= \bigcup_{i=0}^{N-1} C_i$,
				\[
					\lambda \bigl((G_{m+1} \setminus G_{m+1,t}) \cap G_{m,s} \bigr) < \epsilon/2,
				\]
				because every cube enumerated into $G_{m}$ after stage $s$ is disjoint from $G_{m,s}$ by construction.
				Recall that by choice of $s$ we have $\leb (G_m - G_{m,s} ) < \eps$. Therefore
				\bc
				$ \lambda(G_{m+1} \setminus G_{m+1,t})
				\le \eps + \lambda\bigl((G_{m+1} \setminus G_{m+1,t}) \cap G_{m,s}\bigr)
				\le 2 \epsilon$, \ec
				as desired.

				\begin{claim}
					The function $g$ defined in (\ref{eqn:g_function_def}) is $L_1$-computable.
				\end{claim}

				\n
				By Claim~\ref{cl:Gm_comp_measure} and Fact~\ref{fa:1G}, the function $1_{G_i}$ is $L_1$-computable uniformly in~$i$.
				Thus, the function $h_k= \sum_{i=0}^m (-1)^i1_{G_i}$ is also $L_1$-computable uniformly in~$m$.
				It now suffices to show that given a rational $\eps>0$, we can compute $k$ such that the $1$-norm of the rest of the sum is less than $\eps$.

				For each $r \in \NN$ we have
				\bc
					$\vectornorm {\sum_{m=r}^\infty (-1)^m 1_{G_m}}_1 \le \sum_{m=r}^\infty \vectornorm{1_{G_m}}_1 \le \sum_{m=r}^\infty (\leb G_m) \le \tp{-r+1}$.
				\ec
				This shows the claim.
				\begin{claim}
					Let $C_m$ be the dyadic cube enumerated into $G_m$ such that $z \in C_m$. Then the sequence $(C_m)\sN m$ is as required in the lemma.
				\end{claim}
				\n 	First we show that
				\begin{equation} \label{eqn:lim_upper_1} \lim_{m \, \text{even}, m \to \infty} \frac{\int_{C_m} g}{ \leb C_m } =1. \end{equation}
					For $i \le m$ we have $C_m \sub G_i$. Hence
					\[
						\frac{\int_{C_m} \sum_{i=0}^m (-1)^i 1_{G_i}(x)}{ \leb C_m } =1.
					\]
					Now consider $i > m$. Note that by the choice of $r$ in (\ref{eqn:smaller_cubes}) we have $\leb (G_i\cap C_m) \le \tp{-(i-m)m} \leb C_m$.
					Therefore \bc $|\int_{C_m} \sum_{i=m+1}^\infty (-1)^i 1_{G_i}(x)| \le
				\leb C_m \sum _{i=m+1}^\infty \tp{-(i-m)m} \le \tp{-m+1} \leb C_m$. \ec
				This yields (\ref{eqn:lim_upper_1}).
				In a similar way one shows that $\lim_{m \, \text{odd}, m \to \infty} \frac{\int_{C_m} g}{ \leb C_m } =-1$.
				This establishes the claim, and the theorem.
			\end{proof}

			\n
					By Fact \ref{referee}, the function $g$ constructed above is actually $L_p$-computable for each computable real $p\ge 1$. 	For $p > 1$, recall the $p$-variation norm and the Riesz classes $A_p[0,1]$ from Subsection~\ref{ss:variation_norm}.
			Let $A_1[0,1]= AC_0[0,1]$ be the space of absolutely continuous functions.
			\begin{cor}
				Let $p \ge 1$ be a computable real. The following are equivalent for a real $z \in [0,1]$:
				\bi
					\item[(i)]   $z$ is Schnorr random.
					\item[(ii)]  Every function in $A_p[0,1]$ that is computable in the $p$-variation norm is differentiable at $z$.
					\item[(iii)] Every Lipschitz function $f$ that is computable in the $p$-variation norm is differentiable at $z$.
				\ei
			\end{cor}

			\begin{proof}
				(i)$\to$(ii) follows by Theorem~\ref{thm:Lp_Schnorr} and the computable isometry $\+ L_p[0,1] \to A_p[0,1]$ in Subsection~\ref{ss:variation_norm}.
				For (ii)$\to$(iii) it suffices to note that every Lipschitz function is in $A_p[0,1]$.
				Finally, for (iii)$\to$(i), suppose that $z$ is not Schnorr random. Note that by  Fact \ref{referee}, the bounded function $g$ built in the proof of Theorem~\ref{thm:Lp_Schnorr} is $L_p$ computable. The image of this   function  under  the same isometry is Lipschitz and not differentiable at $z$.
			\end{proof}

	\section{Discussion and open problems}
		We have seen that the study of effective Lipschitz functions $f$ is intimately connected to the study of left-r.e.\ bounded martingales and computable (signed) martingales.

		Nondifferentiability of $f$ at a real $z$ corresponds is the conceptual analogue of
		oscillation of the martingale $M(\sss) = S_f(0.\sss, 0.\sss + \tp{-\sssl})$ on the binary expansion of $z$; that is,
		for some $\alpha < \beta$, we have infinitely many initial segments where the value is less than $ \alpha$, and infinitely many where the value is greater than $\beta$.
		We make some points regarding the connection between non-differentiability and oscillation.

		\vsps

		\n 1. It can happen that $f'(z)$ fails to exist even if $M(Z)$ does not oscillate,
		because the martingale only looks at the slope for basic dyadic intervals $[0.\sss, 0.\sss + \tp{-\sssl}]$ containing~$z$,
		while for differentiability we need to consider arbitrary small intervals containing~$z$.
		For instance, following \cite[Section 4]{Brattka.Miller.ea:nd},
		the nondecreasing Lipschitz function $f_0(x)= x \sin (2 \pi \log_2 |x|) +10x$ satisfies $f(x)=10x$ for each $x$ of the form $\pm \tp{-n}$, but $9= \ul Df(0) < \ol Df(0)=11$.
		Let $f$ be the right-shift by $1/2$ of $f_0$. Then $f$ is as required for $z=1/2$.

		\vsps

		\n 2. It is easy to show that if a bit sequence $Z$ is not \ML\ random, then some unbounded left-r.e.\ martingale $M$ oscillates on $Z$:
		take a left-r.e.\ martingale $L$ that succeeds on $Z$.
		Each time $M$ has passed $2$, it ensures the capital decreases to $1$ upon the next bit $1$.
		Since $Z$ has infinitely many $1$'s, $M$ will oscillate between $1$ and values of at least~$2$.

		\vsps

		\n 3.
		At the end of Subsection~\ref{ss:interval-r.e.} we gave an example of an interval-r.e.\ Lipschitz function $f_S$ that is not differentiable at $\Om$. For another example,
		let $P\sub [0,1]$ be an effectively closed class such that $v= \min P$ is \ML\ random, and define an interval-r.e.\ function with Lipschitz constant~$1$ by
		$f(x) = \leb ([0,x]\setminus P)$;
		then it is easy to see that the corresponding left-r.e.\ martingale oscillates on the binary expansion of~$v$, using that $v$ is Borel normal.

		\vsps

		\n 4. The work~\cite{Bienvenu.Greenberg.ea:preprint} shows that a randomness notion of a real $z$ slightly stronger than \ML's,
		called by the authors ``Oberwolfach randomness'', suffices to ensure that each interval-r.e.\ function (not necessarily Lipschitz) is differentiable at~$z$.
		
			\vsps
		
\n 5. A ML-random  real $z$ is called density random if each effectively closed class $\+ P \sub [0,1] $ with $z \in \+ P$ has Lebesgue density $1$ at $z$. Andrews, Cai, Diamondstone, Lempp and Miller  in unpublished work (2012) have shown that this randomness notion is equivalent to    non-oscillation of left-r.e.\ martingales (see \cite{LogicBlog:13}).   Nies  \cite{Nies:STACS}  has shown that $z$ is density random if and only if  all  interval-r.e.\  functions are differentiable at $z$. 
 
	Oberwolfach randomness implies density randomness as shown in \cite{Bienvenu.Greenberg.ea:preprint}.  It is  unknown whether the converse holds. 	In our Theorem~\ref{thm:left-r.e._Lipschitz_variation}
		we represented any non-decreasing interval-r.e.\ Lipschitz function $f$ mapping $0$ to $0$ as the variation $V_g$ of a computable Lipschitz function.
		Even though there is no direct connection between differentiability of $g$ and of $f$ at a real $z$, this result may be helpful in resolving the question.

		\vsps

		An interesting  question is whether an effective version of Rade\-macher's theorem holds:
		can we extend Theorem~\ref{thm:comprd_Lipschitz} to higher dimensions? This would mean that
		\begin{itemize}
		\item[] $z\in [0,1]^n$ is computably random $\LR$
		each computable Lipschitz function $f \colon \, [0,1]^n \ria \mathbb R$ is differentiable at~$z$.
		\end{itemize}
		We conjecture that the answer is yes. By work of Galicki, Nies and Turetsky available at \cite{LogicBlog:11},
		weak 2-randomness of $z$ ensures differentiability at~$z$ of each computable a.e.\ differentiable function.

	\section*{Acknowledgments}
		This material is based upon work supported by
			the National Science Foundation of the USA under Grants No.\ 0652669 and 0901020,
			the Marsden fund of New Zealand under Grant No.\ 08-UOA-187,
			a grant from the John Templeton Foundation, and
			NUS Grant R252-000-420-112.
			The opinions expressed in this publication are those of the authors and do not necessarily reflect the views of the John Templeton Foundation.
		Part of this work was done while F.~Stephan was invited to the University of Auckland in February 2012.

		We would like to thank Jason Rute for very helpful comments and corrections,
		and for providing the argument leading to Theorem~\ref{thm:Rute-extension}.
We would also like to thank the anonymous referees for numerous useful suggestions,
and for providing Fact~\ref{referee} and its proof.
 
	\bibliographystyle{amsnomr}
	\newcommand{\etalchar}[1]{$^{#1}$}
	\providecommand{\bysame}{\leavevmode\hbox to3em{\hrulefill}\thinspace}
	\providecommand{\MR}{\relax\ifhmode\unskip\space\fi MR }
	\providecommand{\MRhref}[2]{%
	\href{http://www.ams.org/mathscinet-getitem?mr=#1}{#2}
	}
	\providecommand{\href}[2]{#2}

\end{document}